\documentclass[]{amsart}

\usepackage{amssymb}
\usepackage{mathrsfs}
\usepackage[applemac]{inputenc}
\usepackage{amsmath, graphicx}
\usepackage{wrapfig}
\usepackage{enumerate}
\usepackage{url}

\def\qmod#1#2{{\hbox{}^{\displaystyle{#1}}}\!\big/\!\hbox{}_{
\displaystyle{#2}}}

\def\resto#1#2{{
#1\hskip 0.4ex\vline_{\hskip 0.2ex\raisebox{-0,2ex}
{{${\scriptstyle #2}$}}}}}

\def\C{{\mathbb C}}
\def\H{{\mathbb H}}

\def\R{{\mathbb R}}

\newcommand{\cal}{\mathcal}
\def\textlmap#1{\mathop{\vbox{\ialign{
                                  ##\crcr
      ${\scriptstyle\hfil\;\;#1\;\;\hfil}$\crcr
      \noalign{\kern-1pt\nointerlineskip}
      \leftarrowfill\crcr}}\;}}
 
\def\ag{{\mathfrak a}}
\def\g{{\mathfrak g}}
\def\hg{{\mathfrak h}}
\def\kg{{\mathfrak k}}
\def\sg{{\mathfrak s}}
\def\ug{{\mathfrak u}}
\def\vg{{\mathfrak v}}
\def\lg{{\mathfrak l}}

\def\Fg{{\mathfrak F}}
\def\Pg{{\mathfrak P}}

\newtheorem{sz}{Satz}
\newtheorem{thry}[sz]{Theorem}
\newtheorem{pr}[sz]{Proposition}
\newtheorem{re}[sz]{Remark}
\newtheorem{co}[sz]{Corollary}
\newtheorem{dt}[sz]{Definition}
\newtheorem{lm}[sz]{Lemma}

\def\End{\mathrm {End}}
\def\Aut{\mathrm {Aut}}
\def\SO{\mathrm {SO}}
\def\GL{\mathrm {GL}}
\def\Hom{\mathrm{Hom}}
\def\id{ \mathrm{id}}
\def\im{\mathrm{im}}

\def\Ad{\mathrm {Ad}}
\def\ad{{\rm ad}}

\begin{document}

\title[Invariant connections, invariant holomorphic bundles]{Invariant
connections and invariant holomorphic bundles on homogeneous manifolds} 

\author{Indranil Biswas}

\address{School of Mathematics, Tata Institute of Fundamental
Research, Homi Bhabha Road, Bombay 400005, India}

\email{indranil@math.tifr.res.in}

\author{Andrei Teleman}

\address{CMI, LATP, Aix-Marseille Universit\'e, 39 Rue F. Joliot-Curie,
F-13453 Marseille Cedex 13, France}

\email{teleman@cmi.univ-mrs.fr}



\begin{abstract}
Let $X$ be a differentiable manifold endowed with a transitive action
$\alpha:A\times X\longrightarrow X$ of a Lie group $A$. Let $K$ be a Lie group.
Under suitable technical assumptions, we give explicit classification theorems, in
terms of explicit finite dimensional quotients, of three classes of objects:
\begin{enumerate}
\item equivalence classes of $\alpha$-invariant $K$-connections on $X$,
\item $\alpha$-invariant gauge classes of $K$-connections on $X$, and
\item $\alpha$-invariant isomorphism classes of pairs $(Q,P)$ 
consisting of a holomorphic $K^\C$-bundle $Q\longrightarrow X$ and a 
$K$-reduction $P$ of $Q$ (when $X$ has an $\alpha$-invariant 
complex structure).
\end{enumerate}

\end{abstract}

\thanks{The authors are grateful to the referee for his comments which contributed to improve the quality of the paper.}
 \thanks{The second author has been partially supported by the ANR project MNGNK, decision 
Nr.  ANR-10-BLAN-0118}

\maketitle

\section{Introduction. Invariant gauge classes of connections}\label{intro}

Hermitian holomorphic bundles on the upper half-plane, which are ${\rm SL}(2,
\R)$-invariant up to isomorphism, can be classified \cite{Bi1}. Our starting
point is the observation of the second author that this problem 
can be reformulated and generalized using ideas from gauge theory.
 
Let $X$ be a connected manifold, $A$ a {\it connected} Lie group, and
$$\alpha:A\times X\longrightarrow X$$ a smooth action. For
any $a\in A$, the diffeomorphism $x \longmapsto
\alpha(a,x)$ of $X$ will be denoted by $f_a$. Let $K$ be a Lie group 
with Lie algebra $\kg$, and let $p:P\longrightarrow X$ be a principal 
$K$-bundle over $X$. Since $A$ is connected, it follows that
$f_a^*(P)\simeq P$ for every $a\in A$. This follows from the homotopy invariance of  pull-backs (see \cite{Hu} Theorem 9.9). Therefore the isomorphism type of $P$ is $\alpha$-invariant.

\begin{dt} \label{inv}
{\rm A connection $\Gamma$ on $P$ will be called} gauge 
$\alpha$-invariant {\rm if for every $a\in A$ there is an 
$\id_X$-covering principal $K$-bundle isomorphism 
$\phi:P\longrightarrow f_a^*(P)$ such that 
$\phi^*(f_a^*(\Gamma))=\Gamma$.}
\end{dt}

The above condition depends only on the image of $A$ in the 
diffeomorphism group $\Aut(X)$, so $\Gamma$ is gauge 
$\alpha$-invariant if and only if it is invariant with respect to 
the natural action of the image of $A$ in this group. Note that 
an $\id_X$-covering bundle isomorphism $P\longrightarrow 
f_a^*(P)$ can be regarded as an $f_a$-covering bundle isomorphism 
$P\longrightarrow P$. 

There exists an interesting alternative interpretation of this 
definition using ideas from gauge theory: Let $p:P 
\longrightarrow X$ be an arbitrary principal $K$-bundle on a manifold $X$ and ${\cal K}_P$ 
its gauge group, i.e., the group $\Aut(P)$ of $\id_X$-covering 
bundle automorphisms of $P$ commuting with the action of $K$. 
This group can be identified with the group 
of $C^\infty$ sections of the group bundle $P\times_\Ad K$, where 
$$\Ad:K\longrightarrow \Aut(K)$$ is the adjoint action 
$(k,l)\longmapsto klk^{-1}$ of $K$ on itself.

We recall that two connections $\Gamma$, $\Gamma'$ on $P$ are 
called gauge equivalent if there exists an $\varphi\in{\cal 
K}_P$ such that $\Gamma'=\varphi^*(\Gamma)$, meaning they are 
conjugate modulo the natural action of the gauge group ${\cal 
K}_P$ on the affine space ${\cal A}(P)$ of connections on $P$. 
The moduli space of connections on $P$ is the quotient
$${\cal B}(P)\,:=\,{\cal A}(P)/{\cal K}_P\ ,
$$
endowed with the quotient topology of the ${\cal C}^\infty$-topology on the affine space ${\cal A}(P)$ \cite{DK}. 
Note that if two principal $K$-bundles $p:P\longrightarrow X$ and
$q:Q\longrightarrow X$ are 
isomorphic (as principal $K$-bundles over $X$), then there exists a 
{\it canonical} identification $${\cal B}(P)\,=\,{\cal B}(Q)\, ,$$ because 
any two $\id_X$-covering isomorphisms $P\longrightarrow Q$ differ 
by the 
composition with a gauge transformation of $P$. Therefore, we can 
define in a coherent way the moduli space ${\cal B}(\Pg)$, where 
$\Pg$ is an isomorphism class of principal $K$-bundles over $X$. Formally ${\cal B}(\Pg)$ is the disjoint union $\coprod_{P\in\Pg}{\cal B}(P)$ factorized by the equivalence
relation generated by the canonical identifications ${\cal B}(P)={\cal B}(Q)$ mentioned above.

Since $f_a^*(P)\simeq P$ for every $a\in A$, the 
isomorphism type $\Pg$ of $P$ is a fixed point under that natural 
action of $A$ on the set of isomorphism types of $K$-bundles on 
$X$. So the moduli space ${\cal B}(\Pg)$ associated with the 
isomorphism type of $P$ comes with a well-defined right $A$-action given by
$$([\Gamma],a)\,\longmapsto\, [f_a^*(\Gamma)]\,\in\, {\cal 
B}(f_a^*(P))\,=\,{\cal B}(\Pg) \, ,$$
where $\Gamma\,\in\,{\cal A}(P)$. Definition \ref{inv} can be reformulated as follows:

\begin{re}\label{rem} {\rm A connection $\Gamma$ on $P$ is 
gauge $\alpha$-invariant if and only if its gauge class 
$[\Gamma]\in{\cal B}(P)\,=\,{\cal B}(\Pg)$ is a fixed point with
respect to the natural $A$-action on ${\cal B}(\Pg)$.}
\end{re}

Our first goal is to describe the set of all classes of gauge 
$\alpha$-invariant connections on $K$-bundles on $X$. More precisely let ${\cal B}_{K}(X)$
be the union
$${\cal B}_{K}(X)\,:=\,\coprod_{\substack{\Pg \ {\rm isomorphism\ type}\\
{\rm of\ } K-{\rm bundles\ on\ }X }} {\cal B}(\Pg)\ ,
$$
so it is the {\it moduli space of all gauge equivalent classes of connections in
$K$-bundles on $X$}. This moduli space comes with a natural $A$-action given by
$$
 {\cal B}(\Pg)\times A\,\ni\, ([\Gamma],a)\longmapsto 
[f_a^*(\Gamma)]\,\in\, {\cal B}(f_a^*(\Pg))\ .
$$

Using this formalism, we see that our first goal is to describe 
the space of $A$-invariant elements in ${\cal B}_{K}(X)$, in 
other words, to describe the fixed point set:
\begin{equation}\label{e1}
{\cal F}_{\alpha,K}\, :=\, \{\gamma\in{\cal B}_{K}(X)|\ 
f_a^*(\gamma)=\gamma\ \forall\, a\in A\}\subset{\cal B}_{K}(X)\ .
\end{equation}

The elements of the above set correspond bijectively to the 
equivalence classes of pairs $(P,\Gamma)$, where $P$ is a 
principal $K$-bundle on $X$, and $\Gamma$ is a gauge 
$\alpha$-invariant connection on $P$; the equivalence relation 
is defined by $\id_X$-covering bundle isomorphisms. The elements of this set will be called $\alpha$-{\it invariant gauge classes of $K$-connections on $X$.} When $A$ is a subgroup of the diffeomorphism group $\Aut(X)$ (or, equivalently, when
the action $A$ is effective), we will also say $A$-invariant gauge classes of
$K$-connections on $X$.

We will study this problem in detail in the particular case 
where $K$ is compact and $A$ acts transitively on $X$. 

Our second goal concerns the case when $X$ has an $A$-invariant complex
structure. In this case we will be interested in the subset
${\cal F}_{\alpha,K}^{1,1}\,\subset\, {\cal F}_{\alpha,K}$
of $\alpha$-invariant gauge classes of {\it type} $(1,1)$ connections on 
$K$-bundles on $X$. We recall that, for a compact Lie group $K$, a
connection on a principal $K$-bundle $P$ on a complex manifold 
$X$ is of type $(1,1)$ if its curvature $F_\Gamma$ is of Hodge type $(1,1)$. 
This condition is equivalent to the condition that the almost 
complex structure on the complexified bundle $Q\,:=\,P\times_K K^\C$
associated with $\Gamma$ via the Chern correspondence (see for 
instance \cite[Section 7.1]{LT}) is integrable. In other words, 
denoting by $G$ the complex reductive group $K^\C$, we see that the elements
of ${\cal F}_{\alpha,K}^{1,1}$ correspond bijectively to equivalence classes of
holomorphic principal $G$-bundles $Q$ on $X$ endowed with a $K$-reduction, with the
equivalence relation defined by $\id_X$-covering holomorphic isomorphisms which
respect the $K$-reductions. Equivalently, two pairs $(Q,P)$ and $(Q',P')$ consisting
of holomorphic $G$-bundles endowed with $K$-reductions are considered equivalent
if they are {\it holomorphically isometric}, meaning there is a holomorphic isomorphism
$Q\,\longrightarrow\, Q'$ of principal $G$-bundles that takes $P$ to $P'$.

\section{Equivariant bundles and invariant connections}

Let $\alpha\,:\,A\times X\,\longrightarrow \,X$ be a smooth action of a 
connected Lie group $A$ on a connected smooth manifold $X$. For
$a\,\in\, A$, denote by $f_a\,:\,X\,\longrightarrow \,X$ the diffeomorphism 
$x\,\longmapsto\, \alpha(x,a)$. Let $K$ be a connected Lie group.

\begin{dt}\label{Kalpha}
{\rm A} principal $(K,\alpha)$-bundle {\rm over 
$X$ is a pair $(P,\beta)$ consisting of a principal $K$-bundle 
$p\,:\,P\,\longrightarrow\, X$ on $X$ and an action $\beta\,:\,A\times P
\,\longrightarrow\, P$ such that for every $a\in A$, the corresponding 
diffeomorphism $\beta_a\,:\,P\,\longrightarrow\, P$, $z\,\longmapsto\,
\beta(a,z)$, is an $f_a$-covering isomorphism of principal $K$-bundles.}
\end{dt}

In other words, $\beta$ is an $\alpha$-covering action by 
principal $K$-bundle 
isomorphisms. According to the terminology used in the 
literature, a pair $(P,\beta)$ as in Definition \ref{Kalpha} is 
also called an $\alpha$-equivariant (or an $A$-equivariant) principal
$K$-bundle over the $A$-manifold $(X,\alpha)$.

\begin{dt}\label{iso}
{\rm Let $(P,\beta)$ and $(P',\beta')$ be two principal
$(K,\alpha)$-bundles over $X$. An} isomorphism 
$$(P,\beta)\longrightarrow (P',\beta')$$ {\rm is an 
$\id_X$-covering $K$-bundle isomorphism that commutes with the 
$A$-actions $\beta$ and $\beta'$ on $P$ and $P'$ respectively.}
\end{dt}

\begin{dt}\label{invv} {\rm Let $(P,\beta)$ be a principal 
$(K,\alpha)$-bundle. A connection $\Gamma$ on $P$ is called}
invariant {\rm if $\beta_a^*(\Gamma)\,=\,\Gamma$ for every $a\,\in\, A$.
In this case, we will also say that $\Gamma$ is a $\beta$-invariant connection on $P$.

An} $\alpha$-invariant $K$-connection {\rm on $X$ is a triple 
$(P,\beta,\Gamma)$, where $(P,\beta)$ is a principal 
$(K,\alpha)$-bundle on $X$, and $\Gamma$ is an 
invariant connection on $P$.}
\end{dt}

Let $(P,\beta)$ be a principal $(K,\alpha)$-bundle. Let
${\cal A}(P)$ be the space of all connections on $P$ and
${\cal A}(P)^\beta\, \subset\, {\cal A}(P)$ the subspace of invariant 
connections on $(P,\beta)$. Let $A^1(\ad(P))^\beta$ denote
the $\beta$-invariant $\ad(P)$-valued 1-forms.
We will need
an explicit description of the space ${\cal A}(P)^\beta$, 
supposing that it is non-empty. The difference $\Gamma'-\Gamma$ 
of two invariant connections is an element of 
$A^1(\ad(P))^\beta$. Conversely, for $\Gamma\,\in\,
{\cal A}(P)^\beta$ and $\omega\,\in\, A^1(\ad(P))^\beta$,
clearly $\Gamma+\omega\,\in\, {\cal A}(P)^\beta$.
Therefore, if non-empty, the space 
${\cal A}(P)^\beta$ has a natural affine space structure
over $A^1(\ad(P))^\beta$. We recall that, denoting by $\kg$ the Lie algebra of $K$, the space $A^1(\ad(P))$ 
can be identified with the space $A^1_\ad(P,\kg)$ of $\kg$-valued 
tensorial 1-forms of type $\ad$ on $P$ (defined in \cite[p. 
75]{KN}); see \cite[p. 76, Example 5.2]{KN} for this 
identification.

Fix now $x_0\,\in\, X$ and denote by $H_0$ the stabilizer of $x_0$ in 
$A$. Choose $y_0\,\in \,P_{x_0}$, and note that $H_0$ acts on the 
fiber $P_{x_0}$ via $\beta$. We define the map 
$\chi_{y_0}\,:\,H_0\,\longrightarrow\, K$ by
\begin{equation}\label{e2}
\phi_h(y_0)\,:=\,\beta(h,y_0)\,=\,y_0(\chi_{y_0}(h))\ .
\end{equation}
This map is a group homomorphism because for $h$, $k\,\in\, H_0$ 
we have
$$y_0(\chi_{y_0}(kh))\,=\,\beta(kh,y_0)=\beta(k,\beta(h,y_0))\,=\,
\beta(k,y_0(\chi_{y_0}(h)))\,=\,
$$
$$
= \, \beta(k,y_0)(\chi_{y_0}(h))\,=\, y_0(\chi_{y_0}(k))(\chi_{y_0}(h))\, .
$$
It is easy to check that, replacing $y_0$ by $y_0 k$ for an element $k\,\in\, K$, we have
$$\chi_{y_0 k}\,=\,k^{-1}\chi_{y_0} k\,=\,\iota_{k^{-1}}\circ \chi_{y_0}\ ,
$$
so $\chi_{y_0}$ is well defined up to conjugation by an element of $K$.
In this formula we used the notation $\iota_k$ for the inner automorphism defined by $k$.

Let $\eta\,\in \,A^1_\ad(P,\kg)^\beta$ be a $\beta$-invariant tensorial form of 
type $\ad$. Its restriction $\eta_{y_0}$ to $T_{y_0}(P)$ descends 
to a linear map $\eta^{y_0}\,:\,T_{x_0}(X)\,\longrightarrow\,\kg$ given 
by $\eta^{y_0}(\xi)\,:=\,\eta_{y_0}(\zeta)$, where $\zeta$ is an 
arbitrary lift of $\xi$ in $T_{y_0}P$. For a tangent vector 
$\xi\,\in\, T_{x_0}(X)$ and a lift $\zeta\in T_{y_0}(P)$ of $\xi$, 
we have $p_*(\beta_{h*}(\zeta))\,=\,f_{h*}(\xi)$ for all $h\,\in\, H_0$.
So $\beta_{h*}(\zeta)$ (and every right translation of it) is a lift
of $f_{h*}(\xi)$. Since $\eta$ is $\beta$-invariant and $\ad$-tensorial we obtain
$$\eta^{y_0}(\xi)\,=\,\eta_{y_0}(\zeta)\,=\,\eta_{\phi_h(y_0)}(\phi_{h*}(\zeta))\,
=\, \eta_{y_0\chi_{y_0}(h)}(\phi_{h*}(\zeta))\,=
$$
$$
=\,R_{\chi_{y_0}(h)}^*(\eta)(R_{\chi_{y_0}(h)*}^{-1}(\phi_{h*}(\zeta)))\,=\,
\ad_{\chi_{y_0}(h)}^{-1}(\eta_{y_0}(R_{\chi_{y_0}(h)*}^{-1}(\phi_{h*}(\zeta))))\,=$$
$$=\,\ad_{\chi_{y_0}(h)}^{-1}(\eta^{y_0}(f_{h*}(\xi))\ ,$$
where (as in \cite{KN}) $R_k$ stands for the right translation $P\,\longrightarrow\,
P$ associated with an element $k\,\in\, K$. Therefore $\eta^{y_0}$ must satisfy the identity
\begin{equation}
\eta^{y_0}(f_{h*}(\xi))\,=\,\ad_{\chi_{y_0}(h)}(\eta^{y_0}(\xi))\ 
\forall\, \xi\in T_{x_0}(X)\, ,\ \forall\, h\,\in\, H_0\ .
\end{equation}

Composing $\eta^{y_0}$ with the derivative 
$\alpha_{x_0*}\,:\,\ag\,\longrightarrow \,T_{x_0}(X)$ at $e\,\in\, A$ of 
the map $a\,\longmapsto\, f_a(x_0)$ we get a linear map
$$\mu^{y_0}\,:=\,\eta^{y_0}\circ\alpha_{x_0*}\,:\,\ag\, 
\longrightarrow \,\kg
$$
satisfying the properties
\begin{enumerate}
\item $$\resto{\mu^{y_0}}{\hg_0}\,=\,0\, ,$$
\item $$\mu^{y_0}\circ \ad_h\,=\,
\ad_{\chi_{y_0}(h)}\circ \mu^{y_0} \ \forall\, h\,\in\, H_0\, . 
$$
\end{enumerate}

The following result is a consequence of Wang's classification theorem for invariant connections on a principal bundle with respect to a "fibre-transitive" action (see \cite{W}, \cite{KN} Theorem 11.5, Theorem 11.7). We include a short proof for completeness.  

\begin{lm}\label{invtensorial} Let $(P,\beta)$ be a 
$(K,\alpha)$-bundle. Choose $x_0\,\in\, X$, $y_0\,\in\, P_{x_0}$, and 
let $\chi_{y_0}\,:\,H_0\,\longrightarrow\, K$ be the associated group 
morphism. If $\alpha$ is transitive, then the map 
$$s_{y_0}\,:\,\eta\,\longmapsto\, 
\mu^{y_0}\,:=\,\eta^{y_0}\circ\alpha_{x_0*}$$
defines an isomorphism between the space $A^1_{\ad}(P,\kg)^\beta$ of
$\beta$-invariant tensorial 1-forms of type $\ad$ on $P$, and the subspace 
$$S_{y_0}:=\{\mu\in \Hom(\ag,\kg)\,|\ \resto{\mu}{\hg_0}=0 ,\ 
\mu \circ \ad_h =\ad_{\chi_{y_0}(h)}\circ \mu \ \forall\, h\in H_0\}\subset \Hom(\ag,\kg)\ .$$
 
For fixed $x_0$, the space $S_{y_0}$ and the isomorphism 
$s_{y_0}: 
A^1_{\ad}(P,\kg)^\beta\,\longrightarrow\, S_{y_0}$ depend on 
$y_0\in P_{x_0}$ according to the formula
$$S_{y_0k}\,=\,\ad_{k^{-1}}(S_{y_0})\ ,\ s_{y_0k}=\ad_{k^{-1}}\circ s_{y_0}\ \forall\,
k\,\in\, K\ .
$$
\end{lm}
 
\begin{proof}
Given a linear map $\mu\,:\,\ag\,\longrightarrow\, \kg$ satisfying the
above two conditions we obtain easily a linear map 
$\eta_{y_0}\,:\,T_{y_0}(P)\,\longrightarrow\, \kg$ vanishing on the 
vertical tangent space at $y_0$. Using the right 
$K$-equivariance and left $A$-invariance property of the 
tensorial forms of type $\ad$, and the transitivity assumption,
we can extend this form to all 
tangent spaces $T_y(P)$. The two properties 
$$\resto{\mu}{\hg_0}\,=\,0 ,\ \mu \circ \ad_h =\ad_{\chi_{y_0}(h)}\circ \mu \ \forall\,
h\,\in \,H_0
$$
ensure that this extension is well-defined, meaning for 
$y\,\in\, P$, the resulting linear map $T_y(P)\,\longrightarrow\, \kg$ 
does not depend on the representation $y\,=\,\phi_a (y_0) k$, 
with $a\,\in\, A$ and $k\,\in\, K$.
\end{proof}

\begin{re}\label{interpret} {\rm The condition $\resto{\mu}{\hg_0}
\,= \,0$ means that the linear map $\mu\,:\,\ag\,\longrightarrow\,\kg$ 
descends to a linear map $\ag/\hg_0\,\longrightarrow\, \kg$. This map will 
also be denoted by $\mu$. The condition
$$\mu \circ \ad_h \,=\,\ad_{\chi_{y_0}(h)}\circ\mu$$
in the 
definition of the space $S_{y_0}$ has a very natural 
interpretation: it means that $\mu$ is a morphism of $H_0$ 
spaces, where $\ag/\hg_0$ is considered as a $H_0$-space via 
the adjoint representation of the subgroup $H_0\,\subset\, A$, and 
$\kg$ is considered a $H_0$-space via $\ad\circ\chi_{y_0}$.}
\end{re}

\begin{co} \label{finiteaffine}
Let $(P,\beta)$ be a $(K,\alpha)$-bundle. Suppose that $\alpha$ 
is transitive and that $P$ admits a $\beta$-invariant 
connection. Choose $x_0\in X$, $y_0\,\in\, P_{x_0}$, and let 
$$\chi_{y_0}\,:\,H_0\,\longrightarrow\, K$$ be the associated group 
morphism. Then the space ${\cal A}(P)^\beta$ of invariant 
connections on $(P,\beta)$ is naturally an affine space over the
finite dimensional space $S_{y_0}\,\subset\,\Hom(\ag/\hg_0,\kg)$.
\end{co}

Our next goal is the classification of $\beta$-invariant
connections on different bundles of type $(K,\alpha)$ up to 
{\it equivalence}.
 
\begin{dt}\label{equivInConn}
{\rm Two $\alpha$-invariant connections $(P,\beta,\Gamma)$, 
$(P',\beta',\Gamma')$ on $X$ are called} equivalent {\rm if 
there is an isomorphism $\phi\,:\,(P,\beta)\,\longrightarrow\, (P',\beta')$ 
of $(K,\alpha)$-bundles such that $\phi^*(\Gamma')\,=\,\Gamma$.}
\end{dt}

We denote by $\Phi_{\alpha,K}$ the set of isomorphism classes 
of $\alpha$-invariant connections. Since the isomorphism class 
of a $\alpha$-invariant connection $\Gamma$ is preserved by
gauge transformations commuting with $\alpha$, we obtain an 
obvious {\it comparison map}
$$\rho_{\alpha,K}\,:\,\Phi_{\alpha,K}\,\longrightarrow\, {\cal 
F}_{\alpha,K}\ ,\ [P,\beta,\Gamma]\,\longmapsto\, [\Gamma]
$$
which will be used in the next section to understand the 
set ${\cal F}_{\alpha,K}$ in \eqref{e1}. 
Intuitively, the comparison map $\rho_{\alpha,K}$ relates 
{\it equivalence classes of invariant connections} to 
{\it invariant equivalence (gauge) classes of connections. } 
Whereas the right hand set ${\cal F}_{\alpha,K}$ depends only 
on the image of $A$ in $\Aut(X)$, the left hand set 
$\Phi_{\alpha,K}$ depends effectively on the action $\alpha$. 
In particular, replacing $\alpha$ by the induced action 
$\widetilde\alpha$ of the universal cover $\widetilde A$ of $A$ we 
will get a comparison map
$$\rho_{{\widetilde\alpha},K}\,:\,\Phi_{{\widetilde \alpha},K}
\,\longrightarrow\, {\cal F}_{{\widetilde\alpha},K}={\cal F}_{\alpha,K}
$$
which will play an important role in the next section.
\\

We will now give an explicit description of the set
$\Phi_{\alpha,K}$ in the special case where $\alpha$ is 
transitive and the principal $H_0$-bundle $q\,:\,A\,\longrightarrow 
\,X$, $a\,\longmapsto\, \alpha(a,x_0)$, where $H_0$ is the stabilizer
of $x_0$, has an invariant connection. We will see that 
this condition has a simple interpretation and is satisfied for a large class of
interesting examples.

{}From now on, throughout this section, we will suppose that 
$\alpha$ is transitive. We will regard the composition $\lambda\,:\,A\times A
\,\longrightarrow\, A$ as a
left-translation action of $A$ on itself. Note that $(A, 
\lambda)$ is a principal $(H_0,\alpha)$-bundle over $X$, 
i.e., $\lambda$ is an $\alpha$-covering action by bundle 
isomorphisms (see Definition \ref{Kalpha}). This pair should 
be regarded as a {\it tautological} equivariant bundle over 
$X$, because it was constructed using only the pointed manifold 
$(X,x_0)$ and the transitive action $\alpha$ on $X$. 

This tautological equivariant bundle has an important role in 
our constructions, because we will see that {\it any} 
$(K,\alpha)$-bundle $(P,\beta)$ over $X$ (for an arbitrary Lie 
group $K$) can be regarded, in an essentially well defined way, 
as a bundle associated with the tautological equivariant bundle 
$(A,\lambda)$ over $X$. This simple remark will allow us to 
construct invariant connections on every $(K,\alpha)$-bundle 
$(P,\beta)$ over $X$, starting with an invariant connection on 
this tautological equivariant bundle.
 
Let $(P,\beta)$ be a $(K,\alpha)$-bundle over $X$. Choose 
$y_0\,\in\, P_{x_0}$, and consider the homomorphism
$\chi_{y_0}\,:\,H\,\longrightarrow\, K$ in \eqref{e2}.
The map $$\psi_{y_0}\,:\ A\,\longrightarrow\, P$$ given by 
$\psi_{y_0}(a)\,:=\,\beta(a,y_0)$ is an $\id_X$-covering 
principal bundle 
morphism of type $\chi_{y_0}\,:\, H_0\,\longrightarrow\, K$, because 
for $a\,\in\, A$ and $h\,\in\ H_0$ we have
$$\psi_{y_0} (a h)\,=\,\beta(ah,y_0)\,=\,\beta(a,\beta(h,y_0))\,=\, 
\beta(a,y_0\chi_{y_0}(h))$$
$$=\,\beta(a,y_0)(\chi_{y_0}(h))\,=\,\psi_{y_0}(a)(\chi_{y_0}(h))
$$
by the definition of $\chi_{y_0}$. We refer to \cite{KN}
for the concept of bundle morphism and for the transformation
of connections via bundle morphisms.

Suppose that the equivariant tautological bundle $(A,\lambda)$ 
over $X$ has an invariant connection $\Gamma_0$. We obtain 
a connection $(\psi_{y_0})_*(\Gamma_0)$ on $P$ which will be
$\beta$-invariant because, for any $a\in A$ we have
$$
(\phi_a)_*((\psi_{y_0})_*(\Gamma_0))\,=\,
(\psi_{y_0})_*((l_a)_*(\Gamma_0))\,=\,(\psi_{y_0})_*(\Gamma_0)\, .
$$
Moreover, the invariant connection $(\psi_{y_0})_*(\Gamma_0)$ 
on $P$ does not depend on the choice of $y_0\in P_{x_0}$, 
because $\psi_{y_0 k}\,=\,R_k\circ \psi_{y_0}$, so the horizontal 
spaces of the connections $(\psi_{y_0})_*(\Gamma_0)$ and 
$(\psi_{y_0k})_*(\Gamma_0)$ coincide. 

Using Lemma \ref{invtensorial} and Corollary \ref{finiteaffine} 
we obtain result, which gives an alternative proof of Theorem 11.7 \cite{KN}:

\begin{pr} \label{can-id} Suppose that $\alpha$ is transitive. Fix $x_0\,\in\, X$ 
with stabilizer $H_0$, and let $\Gamma_0$ be an
invariant connection on the tautological 
equivariant $H_0$-bundle $(A,\lambda)$. Then 
\begin{enumerate}
\item Any $(K,\alpha)$-bundle $(P,\beta)$ over $X$ has a canonical invariant connection.
\item After choosing a point $y_0\,\in\, P_{x_0}$, there is a
canonical identification between the space ${\cal A}(P)^\beta$ of 
$\beta$-invariant connections on $P$ and the space
$$S_{y_0}\,:=\,\{\mu\in \Hom(\ag/\hg_0,\kg)|\ \mu \circ \ad_h \,=\,
\ad_{\chi_{y_0}(h)}\circ \mu \ \forall\, h\in H_0\}\,\subset\, \Hom(\ag/\hg_0,\kg)\, .$$
\end{enumerate}
\end{pr}

Now we can prove the main result of this section. 

\begin{dt}\label{modulidef}
{\rm We introduce the moduli space ${\cal M}(A,H_0,K)$ by}
$$
{\cal M}(A,H_0,K)\,:=\,
$$
$$
\qmod{\big\{(\chi,\mu)\,\in\,\Hom(H_{0},K)\times 
\Hom(\ag/\hg_0,\kg)\, \big|\ \mu \circ \ad_h
\,=\,\ad_{\chi(h)}\circ \mu \ \forall\, h\in H_0\big\}}{K}\ ,
$$
{\rm where $K$ acts by conjugation on the set of pairs $(\chi,\mu)$.}
\end{dt}

\begin{thry}\label{main}
Suppose that $\alpha$ is transitive. Fix $x_0\,\in\, X$ with 
stabilizer $H_0$, and suppose that the tautological equivariant 
$H_0$-bundle $(A,\lambda)$ over $X$ has a $\lambda$-invariant 
connection $\Gamma_0$. Let $K$ be a Lie group. Let
$\Phi_{\alpha,K}$ be the set of equivalence classes of 
$\alpha$-invariant connections, i.e., the set of triples 
$(P,\beta,\Gamma)$, where $(P,\beta)$ is a $(K,\alpha)$-bundle 
and $\Gamma$ a $\beta$-invariant connection on $P$, up to 
equivalence. There exists a natural bijection
$$C_{x_0}\,:\,\Phi_{\alpha,K}\stackrel{\simeq}{\longrightarrow}\, {\cal M}(A,H_0,K)\, .$$
\end{thry}

\begin{proof} Consider a triple $(P,\beta,\Gamma)$, where 
$(P,\beta)$ is a $(K,\alpha)$-bundle and $\Gamma$ a 
$\beta$-invariant connection on $P$. Choosing a point $y_0\,\in\, 
P_{x_0}$ and using the construction explained above we obtain a 
group morphism $\chi_{y_0}\,:\,H_0\,\longrightarrow\, K$, an equivariant
bundle map $\psi_{y_0}\,:\,(A,\lambda)\,\longrightarrow\, (P,\beta)$ over
$X$, and a $\beta$-invariant connection 
$(\psi_{y_0})_*(\Gamma_0)$. By Corollary \ref{finiteaffine}, the 
difference $\Gamma-(\psi_{y_0})_*(\Gamma_0)$ can be regarded as 
an element $\mu^{y_0}\,\in\, S_{y_0}$. We define the map $C_{x_0}$ by
$$(P,\beta,\Gamma)\,\longmapsto\, (\chi_{y_0},\mu^{y_0})\, .
$$

Using the equivariance properties proved above, we see that 
$(\chi_{y_0},\mu^{y_0})$ is independent of $y_0$ up to 
conjugation. The proof uses the fact that the reference 
connection $(\psi_{y_0})_*(\Gamma_0)$ is independent of $y_0$ 
(this was shown above), so $\Gamma-(\psi_{y_0})_*(\Gamma_0)$ is a
well defined $\beta$-invariant tensorial form of type
$\ad$. Moreover, two equivalent triples $(P,\beta,\Gamma)$,
$(P',\beta',\Gamma')$ define obviously the same element in the quotient
$$\qmod{\big\{(\chi,\mu)\,\in\,\Hom(H_{0},K)\times
\Hom(\ag/\hg_0,\kg)\big|\
\mu \circ \ad_h \,=\,\ad_{\chi(h)}\circ \mu \ \forall\, h\in H_0\big\}}{K} \ .
$$

In order to prove that $C_{x_0}$ is bijective, we will construct 
an inverse map. For a pair $(\chi,\mu)\,\in\,\Hom(H_{0},K)\times\Hom( 
\ag/\hg_0,\kg)$ with $\mu \circ \ad_h \,=\,\ad_{\chi(h)}\circ 
\mu$, we define a triple $(P,\beta,\Gamma)$ by 
$$P\,:=\,A\times_\chi K\ ,\ \beta(a',[a,k])\,:=\,[a'a,k]\ ,\ \Gamma=\psi_*(\Gamma_0)+\eta_\mu\ ,
$$
where $\psi\,:\,A\,\longrightarrow\, P$ is the obvious map defined by 
$\psi(a)\,:=\,[a,e]$, and $\eta_\mu$ denotes the $\beta$-invariant tensorial 1-form of type $\ad$ associated with $\mu$. This map is a morphism of principal bundles of 
type $\chi$ over $X$ (see \cite{KN}), and it is equivariant with 
respect to the left $A$-actions.
\end{proof}

Taking into account Remark \ref{interpret} we get

\begin{re} {\rm The condition on $\mu$ in the definition of the 
moduli space ${\cal M}(A,H_0,K)$ has a natural interpretation: 
$\mu$ is a $H_0$-equivariant linear map 
$\ag/\hg_0\,\longrightarrow \,\kg$ with $\kg$ regarded as a 
$H_0$-space via $\ad\circ\chi$.}
\end{re}

\begin{re} {\rm Suppose that $\alpha$ is transitive. Fix $x_0\in 
X$ with stabilizer $H_0$. Then the map
$$\Gamma\,\longmapsto\,\Gamma_e\,\subset\,\ag$$
defines a bijection between the space of $\lambda$-invariant 
connections on the $H_0$-bundle $A\,\longrightarrow\, X$ and the space of 
$\ad_{H_0}$-invariant complements of $\kg_0$ in $\ag$. In 
particular the tautological $(H_0,\alpha)$-bundle $(A,\lambda)$ over
$X$ admits an invariant connection if the pair $(A,H_0)$ satisfies the condition
\begin{equation}\label{C}
\hbox{ The subalgebra $\hg_0\,\subset\,\ag$ admits an 
$\ad_{H_0}$-invariant complement in $\ag$.}
\end{equation}
If $\hg_0$ has an $\ad_{H_0}$-invariant complement $\sg$ in 
$\ag$, then the tautological $(H_0,\alpha)$-bundle $(A,\lambda)$ 
over $X$ has a unique invariant connection $\Gamma_0$ whose 
horizontal space in $e\in A$ is $\sg$.  This follows directly from 
Proposition XVIII in \cite[p. 285]{GHV}.  The connection 
$\Gamma_0$ corresponding to a complement $\sg$ is flat if and only if $\sg$ is a Lie subalgebra.}
\end{re}

If condition  \eqref{C} is satisfied, Proposition \ref{can-id} applies, and we obtain an alternative proof of  Theorem 11.7 \cite{KN}, which classifies invariant connections with respect to a fibre-transitive action assuming  this condition.

Note that the condition \eqref{C} plays an important role in the 
theory of homogeneous spaces. If it is satisfied, the 
homogeneous space $A/H_0$ (or more precisely the pair $(A,H_0)$) 
is called {\it reductive} \cite[p. 30]{Ya}. The following 
remark shows that this condition is always satisfied when $H_0$ 
is compact, and that it is compatible with covers ${\widetilde 
A}\,\longrightarrow \,A$:

\begin{lm}\label{functoriality}
{\rm Suppose again that $\alpha$ is 
transitive. Fix $x_0\,\in\, X$ with stabilizer $H_0$.
\begin{enumerate}
\item If $H_0$ is compact, then the pair $(A,H_0)$ satisfies the 
condition \eqref{C}.
\item If $(A,H_0)$ satisfies the condition \eqref{C}, and 
$c\,:\,{\widetilde A}\,\longrightarrow \,A$ is a cover of $A$, then
the pair $({\widetilde A}\, ,c^{-1}(H_0))$ 
also satisfies the condition \eqref{C}.
\end{enumerate}}
\end{lm}

\begin{proof}
The first statement is Corollary III p. 286 in \cite{GHV}.

The subgroup $c^{-1}(H_0)$ will be denoted by ${\widetilde H}_0$.
For any ${\widetilde a}\,\in\, \widetilde A$ one has $\ad_{\widetilde 
a}\,=\,\ad_{c({\widetilde a})}$ in 
$\GL(\ag)\,=\,\GL({\widetilde\ag})$. So for any $\ad_{H_0}$-invariant 
complement $\sg$ of $\hg_0$ in $\ag$, the pull-back 
$c_*^{-1}(\sg)$ is an $\ad_{{\widetilde H}_0}$-invariant 
complement of ${\widetilde \hg}_0$ in $\widetilde\ag$. Hence the 
second statement follows.
\end{proof}

\begin{co} \label{ForCov}
Suppose that the action $\alpha$ is transitive. Fix $x_0\in X$, 
and suppose that the pair $(A,H_0)$ satisfies the condition 
\eqref{C}. Then Theorem \ref{main} applies to $\alpha$ and to the 
induced action $\widetilde\alpha\,:\,{\widetilde A}\times 
X\,\longrightarrow\, X$ associated 
with any cover $c\,:\,\widetilde A\,\longrightarrow \,A$. 
Therefore, for any such cover $c$, we get an identification
$$\Phi_{\widetilde\alpha,K} 
\stackrel{\simeq}{\longrightarrow}{\cal 
M}(\widetilde A,{\widetilde H}_0,K)\ ,
$$
where ${\widetilde H}_0\, :=\, c^{-1}(H_0)$.
\end{co} 

In the particular case when ${\widetilde H}_0$ is simply 
connected, the right hand quotient in Corollary \ref{ForCov}
can be described using Lie algebra morphisms $\chi$ instead of Lie group morphisms:

\begin{re}\label{WithAlg}
{\rm If ${\widetilde H_0}\,:=\,c^{-1}(H_0)$ is simply connected, then
$${\cal M}(\widetilde A,{\widetilde H}_0,K)\,=$$
$$\qmod{\big\{(\chi,\mu)\,\in\,\Hom_{\rm Lie Alg}(\hg_0,\kg)\times 
\Hom(\ag/\hg_0,\kg)\,\big|\ \mu \circ \ad_h \,=\,\ad_{\chi(h)}\circ 
\mu \ \forall\, h\in \hg_0\big\}}{K}\ ,
$$
which is the quotient of a $K$-invariant real algebraic affine 
subvariety of the vector space $\Hom(\hg_0,\kg)\times 
\Hom(\ag/\hg_0,\kg)$ on which $K$ acts by linear automorphisms.}
\end{re}

\section{Comparing isomorphism classes of invariant connections 
with invariant gauge classes of connections}

Let $p\,:\,P\,\longrightarrow\, X$ be a principal $K$-bundle on $X$. As 
we have seen in Section \ref{intro}, the elements of the 
gauge group ${\cal G}_P$ of $P$ can be interpreted as sections of 
the group bundle $P\times_{\Ad} K$. The fiber $K_x$ of 
$P\times_\Ad K$ over a point $x\in X$ is identified with the 
group of automorphisms $P_x\,\longrightarrow\, P_x$ which commute 
with the right action of $K$ on $P_x$ (this follows from the 
fact that the group of diffeomorphisms of $K$ commuting with the
right translation action of $K$ on itself is precisely the
left translations). Therefore, $K_x\,\simeq\, K$ 
(unique up to an inner automorphism) for every $x\,\in\, X$, and a 
point $y\,\in\, P_x$
defines an isomorphism $i_y\,:\,K_x\,\longrightarrow \,K$ which is given 
by the formula $y(i_y(g))\,=\,g(y)$ for any $g\,\in \,K_x$.

Since $A$ is connected  the isomorphism type of $P$ is 
$\alpha$-invariant;  let $\Gamma$ be a gauge 
$\alpha$-invariant connection on $P$ (see Definition \ref{inv}). 
We denote by $U$ the stabilizer of $\Gamma$ in the gauge group 
${\cal G}_P$ of $P$. The elements of $U$ correspond bijectively 
to $\Gamma$-parallel sections of the associated bundle 
$P\times_{\Ad} K$ endowed with the connection induced by 
$\Gamma$. This proves that, for any fixed $x\,\in\, X$, the group $U$ 
can be identified with the closed subgroup of $K_x$ consisting of 
the elements that commute with the holonomy group of the 
connection $\Gamma$ on $P$ (the holonomy group is a subgroup of
$K_x$ obtained by taking parallel translations of $P_x$ along 
loops based at $x$). Note that $U$ does not need to be 
connected.

Following \cite{Bi1} we define an extension of $A$ by $U$ by 
$$V:=\{(\phi,a)|\ a\in A,\ \phi:P\longrightarrow P\hbox{ is an $f_a$-covering 
bundle isom. with } \phi^*(\Gamma)=\Gamma\}\ .
$$
It is easy to see that $V$ has a Lie group structure such that 
the natural monomorphism $j\,:\,u\,\longmapsto\, (u,e)$ identifies the 
stabilizer $U$ with a closed subgroup of $V$, and such that the 
natural projection $\pi\,:\,V\,\longrightarrow\, A$ becomes a Lie group 
epimorphism. Therefore we obtain a Lie group exact sequence
\begin{equation}\label{ExSeq}
 1\,\longrightarrow\, U\,\stackrel{j}{\longrightarrow}\, V\,
\stackrel{\pi}{\longrightarrow}\, A\,\longrightarrow\, 1\ .
\end{equation}

\begin{pr} \label{lift}
Suppose that the Lie algebra $\ag$ is semisimple and $A$ is 
simply connected. Then the following statements hold:
\begin{enumerate}
\item There exists a Lie group homomorphism $s\,:\,A\,\longrightarrow 
\,V$ such that $\pi\circ s\,=\,\id_A$.

\item If moreover $K$ is compact and all simple summands of 
$\ag$ are non-compact\footnote{By compact Lie algebra we mean a Lie algebra 
$\g$ which is the Lie algebra of a compact Lie group, or  
equivalently, a Lie algebra which admits an inner product which 
is $\ad$-invariant, in the sense that the endomorphisms $\ad(X)$, 
$X\in\g$ are skew-symmetric (see \cite[p. 194, Theorem 6.6]{HM}).}, then $s$ is unique.
\end{enumerate}
\end{pr}

\begin{proof}
Statement (1): Using the terminology used in
\cite[p. 122, Theorem 24.4]{CE}
 we see that when $\ag$ is semisimple, the 
Lie algebra extension
$$
0\,\longrightarrow\, \ug\,\stackrel{j_*}{\longrightarrow}\, \vg
\,\stackrel{\pi_*}{\longrightarrow}\,\ag\,\longrightarrow\, 0 
$$
associated with \eqref{ExSeq} is inessential. Therefore there 
exists a homomorphism of Lie algebras $\sigma\,:\,\ag\,\longrightarrow 
\,\vg$ such that 
$\pi_*\circ\sigma\,=\,\id_\ag$. If $A$ is simply connected then
$\sigma$ is associated with a group homomorphism 
$s\,:\,A\,\longrightarrow \,V$. This homomorphism clearly satisfies
the condition $\pi\circ s\,=\,\id_A$.
\\
\\
Statement(2): Since $K$ is compact, its closed subgroup $U$ is
also compact. As $A$ is connected, the adjoint representation 
$\Ad$ of $V$ defines via $s$ a group homomorphism 
$r\,:\,A\,\longrightarrow\, \Aut_0(U)$ in the connected component 
$\Aut_0(U)$ of the automorphism group $\Aut(U)$ of $U$. Since
$U$ is a compact Lie group it follows that $L\,:=\,\Aut_0(U)$ is a 
compact Lie group (see \cite[p. 264, Theorem 6.66]{HM}). We 
obtain an induced Lie algebra homomorphism 
$r_*\,:\,\ag\,\longrightarrow \,
\lg$, which must vanish, because the Lie algebra $\lg$ is 
compact
and all the simple summands of $\ag$ are non-compact. Therefore 
the homomorphism $r$ is trivial, implying that the elements of 
$A'\,:=\,s(A)$ commute with the elements of $U'\,:=\,j(U)$.

Therefore, the map $v\,\longmapsto\, v(s\pi(v))^{-1}$ is a group 
homomorphism $\theta\,:\,V\,\longrightarrow\, U$ whose kernel is $A'$. 
For another group homomorphism $s_1\,:\,A\,\longrightarrow\, V$ with 
$\pi\circ s_1\,=\,\id_A$, we 
obtain a group homomorphism $\theta\circ s_1\,:\,A\,\longrightarrow\, U$, 
which is also trivial because $U$ is compact and all the 
simple summands of $\ag$ are non-compact. Therefore 
$\im(s_1)\,\subset\, A'$, which shows that 
$s_1(a)\,=\,(\resto{\pi}{A'})^{-1}(a)\,=\,s(a)$ for every $a\,\in\, A$.
\end{proof}

Note that a group homomorphism $s\,:\,A\,\longrightarrow \,V$ with $\pi\circ s\,=\,\id_A$ 
can be regarded as an action $\beta\,:\,A\times P\,\longrightarrow\, P$ 
by bundle isomorphisms leaving $\Gamma$ invariant. Therefore:
 
\begin{co}
Suppose that $\ag$ is semisimple, and denote by 
$\widetilde\alpha\,:\,\widetilde 
A\times X\,\longrightarrow \,X$ the induced action of the universal 
cover $\widetilde A$ of $A$. Then for every gauge 
$\alpha$-invariant connection $\Gamma\,\in\,{\cal A}(P)$, there 
exists an $\widetilde\alpha$-covering action $\beta\,:\,{\widetilde 
A}\times P\,\longrightarrow \,P$ by bundle isomorphisms such that 
$\Gamma$ is $\beta$-invariant. If, moreover, $K$ is compact and 
all simple summands of $\ag$ are non-compact, the action $\beta$ 
is unique.
\end{co}

In the case when we have uniqueness of action preserving 
$\Gamma$, we will write $\beta_\Gamma$ instead of $\beta$. 

Suppose now that we have two gauge $\alpha$-invariant 
connections $\Gamma\,\in\,{\cal A}(P)$ and $\Gamma'\,\in\,{\cal A}(P')$ and 
an $\id_X$-covering bundle isomorphism $\phi\,:\,P\,\longrightarrow\, P'$ such that 
$\phi^*(\Gamma')\,=\,\Gamma$. Using the uniqueness of action in
Proposition \ref{lift} we see that $\phi$ is equivariant with 
respect to the two actions $\beta_\Gamma$, $\beta_{\Gamma'}$. 
In other words, in this case one can assign in a well defined 
way to every $\alpha$-invariant gauge class $[\Gamma]$ of 
$K$-connections an equivalence class $[P,\beta_\Gamma,\Gamma]$ 
of $\widetilde\alpha$-invariant connections.

Therefore, recalling that the set ${\cal F}_{\alpha,K}$ in 
\eqref{e1} depends only on the image of $A$ in $\Aut(X)$, we 
obtain:

\begin{co} If $\ag$ is semisimple then the comparison map
$$\rho_{\tilde \alpha,K}\,:\,\Phi_{{\widetilde\alpha},K}\longrightarrow 
{\cal F}_{{\widetilde\alpha},K}\,=\,{\cal F}_{\alpha,K}\ , \ 
[P,\beta,\Gamma]\,\longmapsto\, [\Gamma]
$$
is surjective. If, moreover, $K$ is compact and all the simple 
summands of $\ag$ are non-compact, this map is bijective.
\end{co}

Using Corollary \ref{ForCov} and Remark \ref{WithAlg} we obtain:

\begin{co} \label{finalcoro}
Suppose that $\ag$ is semisimple, $K$ is compact, all the simple 
summands of $\ag$ are non-compact, the action $\alpha$ is 
transitive, and the pair $(A,H_0)$ satisfies the condition 
\eqref{C} (which holds automatically when $H_0$ is compact). 
 Then the set $\Phi_{{\widetilde\alpha},K}\,\simeq\,{\cal 
F}_{\alpha,K}$ can be identified with ${\cal 
M}(\widetilde A,{\widetilde H}_0,K)$ (see Definition \ref{modulidef}), where
${\widetilde H}_0\, :=\, c^{-1}(H_0)$. If, 
moreover, the pull-back ${\widetilde H}_0$ is 
simply connected, then the set $\Phi_{{\widetilde\alpha},K}$ can 
be identified with the quotient 
$$\qmod{\big\{(\chi,\mu)\in\Hom_{\rm Lie Alg}(\hg_0,\kg)\times 
\Hom(\ag/\hg_0,\kg)\,\big|\ \mu \circ \ad_h =\ad_{\chi(h)}\circ 
\mu \ \forall\, h\in \hg_0\big\}}{K}\ ,
$$
which is the quotient of a $K$-invariant real algebraic affine 
subvariety of the vector space $\Hom(\hg_0,\kg)\times 
\Hom(\ag/\hg_0,\kg)$ on which $K$ acts by linear automorphisms.
\end{co}

\section{Examples}

\subsection{Invariant connections over the half-plane}

The main 
result in \cite{Bi1} can be recovered as a special case of our general results. The upper half-plane $\H$ can be identified with the homogeneous manifold
$\mathrm{PSL}(2)/H_0$, where $H_0\,=\,\SO(2)/\{\pm1\}$, whose Lie algebra is
$$\hg_0\,=\,\R h_0\ ,\ h_0\,:=\,\left(\begin{matrix} 0&-1\cr 1&0\end{matrix}\right)\ .
$$
 
Since the pull-back ${\widetilde H}_0$ in 
$\widetilde{\mathrm{PSL}}(2)$ is simply connected, we obtain
$${\cal M}(\widetilde{\mathrm{PSL}}(2),{\widetilde H}_0,K)=$$
$$\qmod{\big\{(\chi,\mu)\in\Hom_{\rm Lie Alg}(\hg_0,\kg)\times \Hom(sl(2)/\hg_0,\kg)\big|\ \mu \circ \ad_h =\ad_{\chi(h)}\circ \mu \ \forall\, h\in \hg_0\big\}}{K}\ .
$$

The space $\Hom(sl(2,\R)/\hg_0,\kg)$ can be identified with 
$\Hom(\mathfrak{h}_0^\bot,\kg)$, where $\mathfrak{h}_0^\bot$ is 
the complement
$$\mathfrak{h}_0^\bot\,:=\,\left\langle \left(\begin{matrix} 1&0\cr 0&-1\end{matrix}\right), \left(\begin{matrix}0&1\cr 1& 0\end{matrix}\right)\right\rangle$$
of $\mathfrak{h}_0$, which is $\ad_{H_0}$-invariant. Putting
$$B\,:=\,\mu\left(\begin{matrix} 1&0\cr 0&-1\end{matrix}\right)\ ,\ C\,:=\,\mu\left(\begin{matrix}0&1\cr 1& 0\end{matrix}\right)\ ,$$
we see that the condition $\mu \circ \ad_h \,=\,\ad_{\chi(h)}\circ \mu$, $\forall\, h
\,\in\, \hg_0$ is equivalent to
$$[\chi(h_0), B]\,=\, 2C,\ [\chi(h_0), C]\,=\,-2 B\ ,$$ 
so, denoting $A\,:=\,C+iB\,\in\, \mathfrak{k}\otimes\C$, this can be written as
\begin{equation}\label{eq}
[\chi(h_0), A]\,=\,2i A\, .
\end{equation}

On the other hand $\chi$ is obviously determined by the vector $\chi_0\,:=\,
\chi(h_0)\,\in\,\kg$. Therefore, the moduli space ${\cal M}(\widetilde{\mathrm{PSL}}(2),{\widetilde H}_0,K)$ can be identified with the quotient 
$${\cal M}_K\,:=\,\qmod{\{(\chi_0,A)\,\in\, \kg\times (\kg\otimes\C)|\ [\chi_0,A]\,=\,2iA\}}{K}\ .
$$

Denoting by $\alpha_\H$, ${\widetilde\alpha}_\H$ the standard 
actions 
of $\mathrm{PSL}(2)$, respectively $\widetilde{\mathrm{PSL}}(2)$ on $\H$, and using Corollary \ref{finalcoro}, we obtain

\begin{co} The set $\Phi_{{\widetilde\alpha}_\H,K}$ of 
equivalence 
classes of ${\widetilde\alpha}_\H$-invariant $K$-connections on 
$\H$ 
can be naturally identified with the moduli space ${\cal M}_K$. If $K$ is compact, then the same moduli space classifies 
\begin{enumerate}
\item $\alpha_\H$-invariant gauge classes of $K$-connections,
\item $\alpha_\H$-invariant isomorphism classes of pairs $(Q,P)$ consisting of a holomorphic $K^\C$-bundle $Q$ and a differentiable $K$-reduction $P$ of $Q$. 
\end{enumerate}
\end{co}

\subsection{The case of complex homogeneous complex manifolds}

We begin with the following important

\begin{dt}
{\rm Let $K$ be a compact Lie group and $X$ a complex manifold. A} 
Hermitian holomorphic $K$-bundle {\rm is a pair $(Q,P)$ 
consisting of a holomorphic $K^\C$-bundle $Q\,\longrightarrow\, M$ and a 
$K$-reduction $P$ of $Q$. An isomorphism (or an isometric 
biholomorphic isomorphism) of Hermitian holomorphic $K$-bundles 
$(Q,P)$, $(Q',P')$ is a holomorphic isomorphism $f\,:\,Q\,\longrightarrow\, Q'$ such 
that $f(P)\,=\,P'$.}
\end{dt}

Suppose that the conditions of Corollary \ref{finalcoro} are 
satisfied, $K$ is compact, and $X$ possesses an $A$-invariant 
complex structure. The classification of isomorphism classes of 
Hermitian holomorphic $K$-bundles $(Q,P)$ on $X$ reduces to the classification of 
$\alpha$-invariant gauge classes of $K$-connections which are
of   type $(1,1)$ (defined in Section \ref{intro}). The 
condition of being type $(1,1)$ produces an explicit algebraic equation on the space
of pairs $(\chi,\mu)$ (as in Corollary \ref{finalcoro}). This equation has a simple form when the tautological equivariant $H_0$-bundle $(A,\lambda)$ over $X$ has a $\lambda$-invariant connection $\Gamma_0$, which is itself of  type $(1,1)$. This is the case when $X$ is an irreducible Hermitian symmetric space of non-compact type (see \cite{Bi2}).

Suppose that the pair $(A,H_0)$ is reductive (i.e., it satisfies condition (\ref{C})).
Let $\sg$ be a $H_0$-invariant complement of $\hg_0$ in $\ag$ and $\Gamma_0$ the 
corresponding invariant connection on the $H_0$-bundle $A\,\longrightarrow\, X$. When
$\Gamma_0$ is 
of type $(1,1)$, all the induced connections $(\psi_{y_0})_*(\Gamma_0)$ (see the proof of 
Theorem \ref{main}) will also be of type $(1,1)$. For a pair $(\chi,\mu)\,\in\,\Hom( 
H_{0},K)\times \Hom(\ag/\hg_0,\kg)$ with $\mu \circ \ad_h \,=\,\ad_{\chi(h)}\circ \mu \ 
\forall\, h\in H_0$, the condition that the associated connection is of type $(1,1)$
reduces to the following condition:
\begin{equation}\label{int}\bigg[D_0(\eta_\mu)+
\frac{1}{2}[\eta_\mu\wedge\eta_\mu]\bigg]^{2,0}+
\bigg[D_0(\eta_\mu)+\frac{1}{2}[\eta_\mu\wedge\eta_\mu]\bigg]^{0,2}\,=\,0\ ,
\end{equation}
where $D_0$ is the exterior covariant derivative  \cite{KN} associated with the type $(1,1)$
connection $(\psi)_*(\Gamma_0)$, and $\eta_\mu$ is the $\beta$-invariant
tensorial 1-form of type $\ad$ associated with $\mu$ (see the proof of Theorem \ref{main}).
The tensorial 2-form $D_0(\eta_\mu)+\frac{1}{2}[\eta_\mu\wedge\eta_\mu]$ is determined by its value at $y_0\,=\,[(e,e)]\,\in\, A\times_\chi K$, which is a 
skew-symmetric bilinear map $T_{y_0}(P)\times T_{y_0}(P)\,\longrightarrow\, \kg$ whose
pull-back via $\psi_{y_0}$ is a bilinear map $\ag\times\ag\,\longrightarrow\, \kg$ vanishing when one of the arguments belongs to $\hg_0$. 
\begin{lm}\label{D0}
The restriction to $\sg\times\sg$ of the skew-symmetric bilinear map $\ag\times\ag\,\longrightarrow\, \kg$ induced by  $D_0(\eta_\mu)+\frac{1}{2}[\eta_\mu\wedge\eta_\mu]$ via $ \psi_{y_0}$ is given by
$$(\xi,\zeta)\,\longmapsto\,
-\mu([\xi,\zeta])+[\mu(\xi),\mu(\zeta)]\ .
$$

\end{lm}
\begin{proof} The exterior covariant derivative $D_0(\eta_\mu)$ with respect to the connection  $\psi_*(\Gamma_0)$  of $\eta_\mu$ corresponds to the exterior covariant derivative  with respect to $\Gamma_0$ of the pull-back of $\eta_\mu$ via the bundle morphism $\psi$. This pull-back   is a  tensorial 1-form of type $\chi_*\circ\ad$  on the total space $A$ of the $H_0$-bundle $A\to X$, and coincides with  the  left invariant $\kg$-valued  1-form $\tilde\mu$ on $A$ which extends   $\mu:T_aA=\ag\to\kg$.  It suffices to compute   $(D_{\Gamma_0}\tilde\mu)(\xi,\zeta)$ for two tangent vectors $\xi$, $\zeta\in\sg$.  Let $\tilde\xi$, $\tilde \zeta$ be  the  left invariant vector fields on $A$ determined by $\xi$, $\zeta$.  Since $\tilde \xi$ and $\tilde \zeta$ are $\Gamma_0$-horizontal, and the functions $\tilde\mu(\tilde \xi)$, $\tilde\mu(\tilde\zeta):A\longrightarrow\R$ are constant, we get
$$D_{\Gamma_0}(\tilde\mu)(\tilde\xi,\tilde\zeta)= (d(\tilde\mu))(\tilde\xi,\tilde\zeta)=\tilde \xi(\tilde\mu(\tilde\zeta))-\tilde\zeta(\tilde\mu(\tilde\xi))-\tilde\mu([\tilde\xi,\tilde\zeta])=- \tilde\mu([\tilde\xi,\tilde\zeta])\ . 
$$
This shows that  $D_{\Gamma_0}(\tilde\mu)(\xi,\zeta)=-\mu([\xi,\zeta])$  where $[\xi,\zeta]$ denotes the Lie bracket of $\xi$, $\zeta$ in the Lie algebra $\ag$. Note that  $\mu([\xi,\zeta])$ depends only on the $\sg$-component (the horizontal component) of $[\xi,\zeta]$.

\end{proof}

Using Lemma \ref{D0} and our results about the classification of invariant connections we obtain:

\begin{thry}\label{final}
Let $\sg$ be a $H_0$-invariant complement of
$\hg_0$ in $\ag$ endowed with a complex structure $J\,\in\,\End(\sg)$ such that
\begin{enumerate}
\item The invariant almost complex structure determined by $J$ on $X\,=\,A/H_0$ is
integrable.
\item \label{secondC} The curvature of the connection $\Gamma_0$ on the $H_0$-bundle
$A\,\longrightarrow\, X\,=\,A/H_0$ is of Hodge type $(1,1)$.
\end{enumerate}

Let $K$ be a compact Lie group. Then the equivalence classes of triples $(Q,P,\beta)$ consisting of a holomorphic $K$-bundle $(Q,P)$ on $X$, and an $\alpha$-lifting
action $\beta$ by holomorphic $K$-bundle isomorphisms, correspond bijectively to
the points of the quotient ${\cal M}(A,H_0,K,\sg,J)\,\subset\, {\cal M}(A,H_0,K)$
defined by
$${\cal M}(A,H_0,K,\sg,J)\,:=\,
\bigg\{(\chi,\mu)\in\Hom(H_{0},K)\times \Hom(\ag/\hg_0,\kg)\big|\ \ \ \ \ \ \ \ \ \ \ \ \ \ \ \ \ \ \ \ \ \ \
$$
$$
\ \ \ \ \ \ \ \ \ \ \ \ \ \ \ \ \ \ \ \ \ \ \ \ \ \ \ \ \ \ \ \ \ \ \ \ \ \ \ \ \mu \circ \ad_h =\ad_{\chi(h)}\circ \mu \ \forall\, h\in H_0,\ \Fg_J(\mu)=0\bigg\}/K
$$
where the map $\Fg_J\,:\,\Hom(\ag/\hg_0,\kg)\,\longrightarrow\, \mathrm{Alt}^2(\sg,\kg)$ is defined by
$$\Fg_J(\mu)(\xi,\zeta)\,:=\,-
\mu([\xi,\zeta])+[\mu(\xi),\mu(\zeta)]+\mu([J\xi,J\zeta])-[\mu(J\xi),\mu(J\zeta)]\ .
$$

If, moreover, $\ag$ is semisimple and all the simple summands of $\ag$ are of non-compact type, then the $\alpha$-invariant isomorphism classes of holomorphic $K$-bundles on $X$ correspond bijectively to the points of the quotient ${\cal M}(A,{\widetilde H}_0,K,\sg,J)$, where $c\,:\,{\widetilde A}\,\longrightarrow\, A$ is the universal cover of $A$ and
${\widetilde H}_0\,:=\,c^{-1}(H_0)$.
\end{thry}

The main result of \cite{Bi2} can be recovered as a special case of this general theorem:

\begin{re}
{\rm Taking for $A$ a simple Lie group of non-compact type and 
for $H_0$ a maximal compact subgroup of $A$, we obtain the 
classification of $\alpha$-invariant classes of holomorphic 
$K$-bundles $(Q,P)$ on any irreducible symmetric Hermitian space 
of non-compact type. Condition (\ref{secondC}) required in 
Theorem \ref{final} is satisfied   (see \cite{Bi2}  and Theorem 9.6 in \cite{KN2}) 
so in this case the $\alpha$-invariant equivalence classes of 
pairs $(Q,P)$ correspond bijectively to the moduli space ${\cal 
M}(A,{\widetilde H}_0,K,\sg,J)$.}
\end{re}

\subsection{Non-transitive actions}

Let now $\alpha$ be a smooth action of a Lie group $A$ on a 
manifold $X$. Restricting an $\alpha$-invariant $K$-connection 
$(P,\beta,\Gamma)$ on $X$ to an orbit $Y\,\subset\, X$ of $\alpha$, 
one obtains an $\alpha_Y$-invariant $K$-connection on $Y$ endowed 
with the induced transitive action $\alpha_Y$. 

If $A$ acts on $X$ with compact stabilizers, then Theorem 
\ref{main} can be applied to all these transitive actions, and one obtains for every
orbit $Y$ an explicit description of the set $\Phi_{\alpha_Y,K}$ of equivalence
classes of $\alpha_Y$-invariant $K$-connections on $Y$ in terms of a moduli
space ${\cal M}_Y$, which is a $K$-quotient of a finite dimensional space. A natural
problem is to endow the union ${\cal M}\,:=\,
\coprod_{Y\in X/A} {\cal M}_Y$ with a natural topology such that the projection 
$$
r\,:\,{\cal M}\,\longrightarrow\, X/A
$$ 
is continuous, and such that every $\alpha$-invariant 
$K$-connection $(P,\beta,\Gamma)$ on $X$ defines a continuous 
section of $r$. We believe that a natural strategy for 
understanding the set $\Phi_{\alpha,K}$ of equivalence classes 
of $\alpha$-invariant $K$-connections on $X$ is to study the 
map $R$ which associates to every $\alpha$-invariant 
$K$-connection $(P,\beta,\Gamma)$ on $X$ the section of the 
fibration $r$ obtained by restriction to the orbits. The first 
step in this direction would be to describe explicitly the 
topology of the total space ${\cal M}$, and the image and the 
fibers of $R$.

\end{document}